\documentclass[12pt]{article}

\usepackage{graphicx}
\usepackage{amsmath,amssymb,amsthm}

\usepackage[all]{xy}
\newenvironment{spmatrix}{\left ( \begin{smallmatrix}} {\end{smallmatrix}\right )}
\newenvironment{sbmatrix}{\left [ \begin{smallmatrix}} {\end{smallmatrix}\right ]}

\newtheorem{theorem}{Theorem}[section]
\newtheorem*{theorem*}{Theorem}
\newtheorem{lemma}[theorem]{Lemma}
\newtheorem{proposition}[theorem]{Proposition}

\newtheorem{corollary}[theorem]{Corollary}
\theoremstyle{definition}
\newtheorem*{pf}{\it Proof}{\it}{\rm}
\newtheorem{definition}[theorem]{Definition}{\bf}{\rm}

\newtheorem*{acknowledgements}{Acknowledgements}
\newtheorem{rem}[theorem]{Remark}{\bf}{\rm}
\newtheorem{exm}[theorem]{Example}{\bf}{\rm}

\let\olditemize\itemize
\renewcommand{\itemize}{
\olditemize
\setlength{\itemsep}{1.2pt}
\setlength{\parskip}{3pt}
\setlength{\parsep}{0pt}
}
\let\oldenumerate\enumerate
\renewcommand{\enumerate}{
\oldenumerate
\setlength{\itemsep}{1.2pt}
\setlength{\parskip}{3pt}
\setlength{\parsep}{0pt}
}
\begin{document}
\title{General origamis and Veech groups of flat surfaces
}
  
  \author{Shun Kumagai\thanks{Research Center for Pure and Applied Mathematics, Graduate School of Information Sciences, Tohoku University, Sendai 980-8579, Japan. 
 This work was supported by JSPS KAKENHI Grant Number JP21J12260. 
\newline  email: shun.kumagai.p5@alumni.tohoku.ac.jp
}}{}
\date{}



  \maketitle



\begin{abstract}
In this century, a square-tiled translation surface (an origami) is intensively studied as an object with special properties of its translation structure and its $SL(2,\mathbb{R})$-orbit embedded in the moduli space. 
We generalize this concept in the language of flat surfaces appearing naturally in the Teichm\"uller theory. 
We study the combinatorial structure of origamis and show that a certain system of linear equations realizes the flat surface in which rectangles of specified moduli replace squares of an origami. 
This construction gives a parametrization of the family of flat surfaces with two finite Jenkins-Strebel directions for each combinatorial structure of two-directional cylinder decomposition. 
Moreover, we obtain the inclusion of Veech groups of such flat surfaces under a covering relation with specific branching behavior. 
\end{abstract}

\section{Introduction}
\label{intro}
A (finite) flat surface is an analytically finite Riemann surface $R$ together with an integrable holomorphic quadratic differential $\phi$ on $R$. 
Coordinates defined by local integral by $\sqrt{\phi}$ form an atlas whose any transition map is half-translation, which is defined up to finitely many conical singularities. 
On a flat surface, the Euclidian flat metric lifts, and we can define affine geometry objects such as locally-affine homeomorphisms. 
The Veech group $\Gamma(R,\phi)$ of a flat surface $(R,\phi)$, the group of derivatives of affine self homeomorphisms on $(R,\phi)$, is introduced by Veech \cite{V} in the context of the study of the geodesic flow on a flat surface. 

Earle and Gardiner \cite{EG} reformulated the theory of flat surfaces in terms of the Teichm\"uller spaces. 
The $SL(2,\mathbb{R})$-orbit of a flat surface $(R,\phi)$, the family of affine deformations of $(R,\phi)$ forms a holomorphic geometric disk embedded in the Teichm\"uller space. 
The maximal subgroup of the Teichm\"uller-modular group acting on the embedded disk is the group of affine self homeomorphisms on $(R,\phi)$, which acts on the disk by the M\"obius transformation. 
Thus the family of affine deformations of $(R,\phi)$ is embedded in the moduli space as an analytic orbifold $\mathbb{H}/\Gamma(R,\phi)$. 
On the other hand, the space of holomorphic quadratic differential on $R$ without poles is known to be a cotangent space of the Teichm\"uller space of $R$. 
It is naturally stratified into strata with specified orders of singularities. 
Each stratum is known to be a smooth orbifold of prescribed dimension \cite{BCGGM}.

Veech groups of flat surfaces are studied in terms of combinatorial objects invariant under affine homeomorphisms, such as \cite{S1}, \cite{B}, \cite{Sh}, \cite{ESS}. 
An abelian (formerly known as oriented) origami \cite{HS1} is a typical example of a flat surface with a combinatorial object. 
It is a finite cover of the unit square torus branched over one point, which comes down to combinatorial characterization such as the monodromy. 
More specially, it is often regarded as a translation surface whose any transition map is a translation. 
Schmith\"usen \cite{S1} showed that the universal Veech group of an abelian origami acts automorphically on the free group $F_2$ and its Veech group is a stabilizer under this action. 
Ellenberg, McReynolds \cite{EM} showed a sufficient condition for a group to be the Veech group of an abelian origami. 

In this paper, we study (general) origamis, which are flat surfaces obtained by gluing finite unit squares along edges. 
It is given by the construction that we reglue an abelian origami after inverting some squares. 
A stabilizer also describes the Veech group of origami under a combinatorial action of the universal Veech group. 
An origami also corresponds to a finite cover of the sphere with prescribed branching behavior. 
We show that a certain system of linear equations realizes the flat surface in which rectangles of specified moduli replace squares of an origami. 

Conversely, the flat surface constructed from an origami with compatible moduli list naturally corresponds to its two-directional cylinder decompositions, as the construction is shown by Earle and Gardiner \cite{EG}. 
For each combinatorial structure of two-directional cylinder decompositions, the family of such flat surfaces is parametrized in the quotient of the solution space of a system of linear equations by a finite group. 
Comparing decompositions into general origamis with compatible moduli lists of a flat surface gives an essential condition for the existence of affine self homeomorphism for each derivative. 
This observation leads to the inclusion of Veech groups of such flat surfaces under a covering relation with specific branching behavior. 

For an origami $(R,\phi)$, $\mathbb{H}/\Gamma(R,\phi)$ becomes an algebraic curve defined over a number field, called an origami curve. 
It is shown by M\"oller \cite{M1} that the action of the absolute Galois group $\mathrm{Gal}(\bar{\mathbb{Q}}/\mathbb{Q})$ respects the embedding of origami curves into the moduli spaces. 
He applied this result to another approach to observe the Galois action on the profinite mapping class group $\hat{\Gamma}_{2,0}$ in a profinite tower \cite{HSL} to be compatible with the embedding of origami curves. 
A special example \cite{HS2} of origami is pointed to have a significant  behavior in the moduli space. 

\medskip\medskip
\section{Flat surface}
\label{sec:2}

Let $R$ be a Riemann surface of finite analytic type $(g,n)$ with $3g-3+n>0$. 
\label{sec:2-1}

\begin{definition}
A \textit{holomorphic quadratic differential} (resp.\ a \textit{holomorphic abelian differential}) $\phi$ on $R$ is a tensor on $R$ whose restriction to each chart $(U,z)$ on $R$ is of the form $\phi(z)dz^2$ (resp.\ $\phi(z)dz$) where $\phi$ is a holomorphic function on $U$. \end{definition}
A pair $(R,\phi)$ of Riemann surface $R$ and a holomorphic quadratic differential $\phi$ on $R$ is called a \textit{flat surface}. 
Zeros and poles of $\phi$ are called \textit{singularities} of $(R,\phi)$ and we denote the set of {singularities} of $(R,\phi)$ by $\mathrm{Sing}(R,\phi)$. 
We say that a point $p\in R$ has \textit{order} $\mathrm{mult}_p(\phi)$ if $p\in\mathrm{Sing}(R,\phi)$ and otherwise $0$. 

Let $p_0\in R$ be a regular point of $\phi$ and $(U,z)$ be a chart around $p_0$. 
Then $\phi$ defines a natural coordinate ($\phi$\textit{-coordinate}) $\zeta_\phi(p)=\int^p_{p_0} \sqrt{\phi(z)}dz$ on $U$, on which $\phi=d\zeta_\phi^2$. 
These coordinates form an atlas on $R^*=R\setminus \mathrm{Sing}(R,\phi)$ whose any coordinate transformation is a half-translation $\zeta \mapsto \pm \zeta +c\ (c\in \mathbb{C})$. 
This atlas extends to a singularity of order $m$ with transition of the form $\zeta\mapsto \zeta^{\frac{m}{2}+1}$. 
We consider only integrable flat surfaces i.e.\ having singularities of order no less than $-1$.

\begin{definition}Let $(R,\phi)$, $(S,\psi)$ be flat surfaces of genus $g$. 
\begin{enumerate}
\item For a matrix $A=\begin{spmatrix}a&b\\ c&d\end{spmatrix}\in GL(2,\mathbb{R})$, we denote by $[A]=\begin{sbmatrix}a&b\\ c&d\end{sbmatrix}$ the quotient class of $A$ in $PSL(2,\mathbb{R})$. 
We define the affine map 
$T_A:x+iy\mapsto (ax+cy)+i(bx+dy)$ on the plane.  

\item We say that a branched covering $f:(S,\psi)\rightarrow (R,\phi)$ is \textit{locally-affine} if there exist finite subsets $\mathrm{Sing}(R,\phi)\subset \Sigma_R\subset R$, $\mathrm{Sing}(S,\psi)\subset \Sigma_S\subset S$ such that $f$ is restricted to a covering $f:S\setminus\Sigma_S\rightarrow R\setminus\Sigma_R$ that is locally represented by $z \mapsto T_A(z) + c$ for some $A\in SL(2,\mathbb{R})$ and $c\in \mathbb{R}^2)$ with respect to the natural coordinates. 
A locally-affine biholomorphism is called an \textit{isomorphism}. 
\item For a locally-affine covering $f:(S,\psi)\rightarrow (R,\phi)$, the local derivative $A$ is constant up to a factor $\{\pm \begin{spmatrix}1&0\\ 0&1\end{spmatrix}\}$ independent of coordinates of the flat structures. 
We call $D_f:=[A]\in PSL(2,\mathbb{R})$ the \textit{derivative} of $f$. 
\item Let $\mathrm{Aff}^+(R,\phi)
$ be the group of locally-affine self-homeomorphisms of $(R,\phi)$. 
We call the group $\Gamma (R,\phi)$ of derivatives of all elements in $\mathrm{Aff}^+(R,\phi)$ the 
\textit{Veech group}.
\end{enumerate}
\end{definition}

For a flat surface $(R,\phi)$ and $t\in\mathbb{H}$, we may consider a flat surface $(R,\phi_t)$ such that $\zeta_{\phi_t}=\mathrm{Re}(\zeta_\phi) + t\mathrm{Im}(\zeta_\phi)$. 
The Teichm\"uller theorem states that the natural map $\mathrm{id}_R:(R,\phi)\rightarrow(R,\phi_t)$ is the extremal deformation for the quasiconformal dilatation. 
Thus the map $\hat{\iota}_{(R,\phi)}:\mathbb{H}\rightarrow T(R):t\mapsto [(R,\phi_t), id_R]$ defines an isometric embedding into the Teichm\"uller space $T(R)$ with respect to the hyperbolic metric and the Teichm\"uller metric. 
The maximal subgroup of the Teichm\"uller-modular group acting on $\hat{\iota}_{(R,\phi)}(\mathbb{H})\subset T(R)$ is the affine group $\mathrm{Aff}^+(R,\phi)$, which acts on $\hat{\iota}_{(R,\phi)}(\mathbb{H})$ by the M\"obius transformation 
\begin{equation}
f\cdot\hat{\iota}_{(R,\phi)}(t)=\hat{\iota}_{(R,\phi)}(D_f(t)),\ \ f\in\mathrm{Aff}^+(R,\phi),\  t\in\mathbb{H}.
\end{equation}

Finally, the projected image of $\hat{\iota}_{(R,\phi)}(\mathbb{H})\subset T(R)$ in the moduli space $M(R)$ is understood to be the orbifold $\mathbb{H}/\Gamma(R,\phi)$. See \cite{EG}, \cite{IT} for details. 

\begin{rem}
We say that a flat surface $(R,\phi)$ is \textit{abelian} if $\phi$ becomes the square of an abelian differential on $R$ and otherwise \textit{non-abelian}. 
An abelian flat surface is often thought of as a \textit{translation surface} with an atlas whose any coordinate transformation is of the form $\zeta \mapsto \zeta +c\ (c\in \mathbb{C})$. 
As the derivative of a locally-affine homeomorphism between such flat surfaces is constant, it is defined as a matrix in $SL(2,\mathbb{R})$. 
The Veech group of such a flat surface is defined as a subgroup of $SL(2,\mathbb{R})$. 
\end{rem}	

For a flat surface $(R,\phi)\in \mathcal{Q}_g$, the analytic continuation of the two branches of $\sqrt{\phi}$ induces 
a flat surface $(\hat{R},\sqrt{\phi})$ on which $\sqrt{\phi}$ is a globally defined abelian differential. 
In abelian case, $(\hat{R},\sqrt{\phi})$ is the disjoint union of two copies of $(R,\phi)$. 
Otherwise, $(\hat{R},\sqrt{\phi})$ is the minimal abelian flat surface that covers $(R,\phi)$ called the \textit{canonical double cover}. 
See \cite[Construction 1]{Lanneau} for details. 

\medskip\medskip
\section{Origami}
\label{sec2-2}

\begin{definition}
An \textit{origami} of degree $d$ is a flat surface obtained from $d$ copies of the Euclidian unit squares by gluing along edges. 
\end{definition}

Figure \ref{origami_exm} shows an example of a non-abelian origami. 
In general, an origami of degree $d$ is a $2d$-fold cover of the sphere over four points whose {valency list} is of the form $(2^d\mid 2^d\mid 2^d\mid *\ )$. 
That is, any critical point over the three points has multiplicity two and so is nonsingular. 
The rest branched point pulls back the singularities of the origami. 
\begin{figure}[htbp]
\begin{center}
\includegraphics[width=120mm]{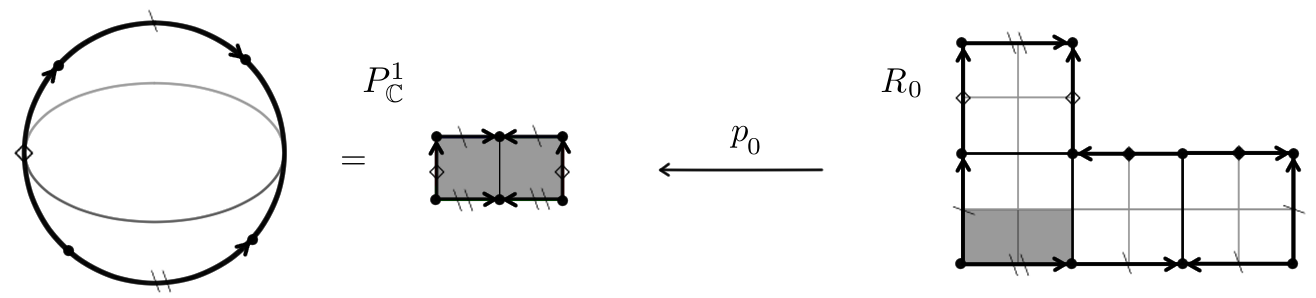}
  \caption{An origami of degree $4$: edges with the same character are glued so that the arrows match. 
It is an $8$-fold covering of the sphere $P^1_\mathbb{C}$ with valency list $(2^4\mid 2^4\mid 2^4\mid 1^2,3^2)$. }
\label{origami_exm}    
\end{center}
\end{figure}

An abelian origami of degree $d$ admits a $d$-fold cover of the unit square torus branched over just one point that pulls back  singularities of origami.
It leads to several combinatorial characterizations as follows.

\begin{lemma}[{\cite[Proposition 5.1]{HS1}}]\label{origami}
Let $S_d$ be the symmetry group order $d$. 
Then, an abelian origami of degree $d$ is up to equivalence uniquely determined by each of the following. 
\begin{enumerate}
\item{A $d$-fold cover $p:R\rightarrow E$ of the unit square torus branched at most over one point. }
\item {A connected, oriented graph $(\mathcal{V, E})$ with $|\mathcal{V}|=d$ such that each vertex has precisely one incoming edge and one outgoing edge labeled one with $x$ and one with $y$, respectively. }
\item{A pair of two permutations $x,y\in S_d$ generating a transitive group. }
\item{A subgroup $H$ of the free group $F_2$ of index $d$. }
\end{enumerate}
\end{lemma}

Let $T=\begin{spmatrix}1&1\\0&1\end{spmatrix}$ and $S=\begin{spmatrix}0&1\\-1&0\end{spmatrix}$. 
Note that $T$ and $S$ generate $SL(2,\mathbb{Z})$.  

\begin{proposition}[{\cite[Lemma 2.8]{S1}}]\label{diagram} 
Suppose that $SL(2,\mathbb{Z})$ acts automorphically on $F_2$ by the formula $T\begin{spmatrix}x\\y\end{spmatrix}=\begin{spmatrix}x\\xy\end{spmatrix}$ and $S\begin{spmatrix}x\\y\end{spmatrix}=\begin{spmatrix}y^{-1}\\x\end{spmatrix}$.  
Then the Veech group of an abelian origami $H<F_2$ is the stabilizer of the conjugacy class of $H$ under this action of $SL(2,\mathbb{Z})$. 
\end{proposition}

 Let $I_d=\{1,\ldots d\}$ be the set of $d$ indices and $\bar{I}_d=\{\pm1,\ldots \pm d\}$ be its double. 
Let $\bar{S}_d$ be the group of odd functions in $\mathrm{Sym}\bar{I}_d$ that naturally embeds $S_d$ as the group of sign-preserving functions. 
For each $x\in S_d$ and $\varepsilon\in\{\pm1\}^d$, let $x^\varepsilon$ denote the map defined by
\begin{equation}
x^\varepsilon(\pm\lambda) = \left\{
\begin{array}{ll}
\pm x(\lambda)&\text{ if }\pm\varepsilon(\lambda)=+1\\
\pm x^{-1}(\lambda)&\text{ if }\pm\varepsilon(\lambda)=-1 
\end{array}
\right. 
\ \text{for each }\lambda\in {I}_d.
\end{equation}
\begin{lemma}[{\cite[Section 4]{K}}]\label{cutpaste}
An origami of degree $d$ is given by an abelian origami $( x,y)$ of degree $d$ with $\varepsilon\in\{\pm1 \}^d$ as a flat surface whose canonical double cover is the abelian origami $(\mathbf{x,y})=(x^\mathrm{sign}, \varepsilon y^\varepsilon\varepsilon(y^\varepsilon))$. 
\end{lemma}
The definition of $(\mathbf{x,y})$ corresponds to the construction of an origami shown in Figure \ref{construction}. 
\begin{figure}[htbp]
\begin{center}
\includegraphics[width=110mm]{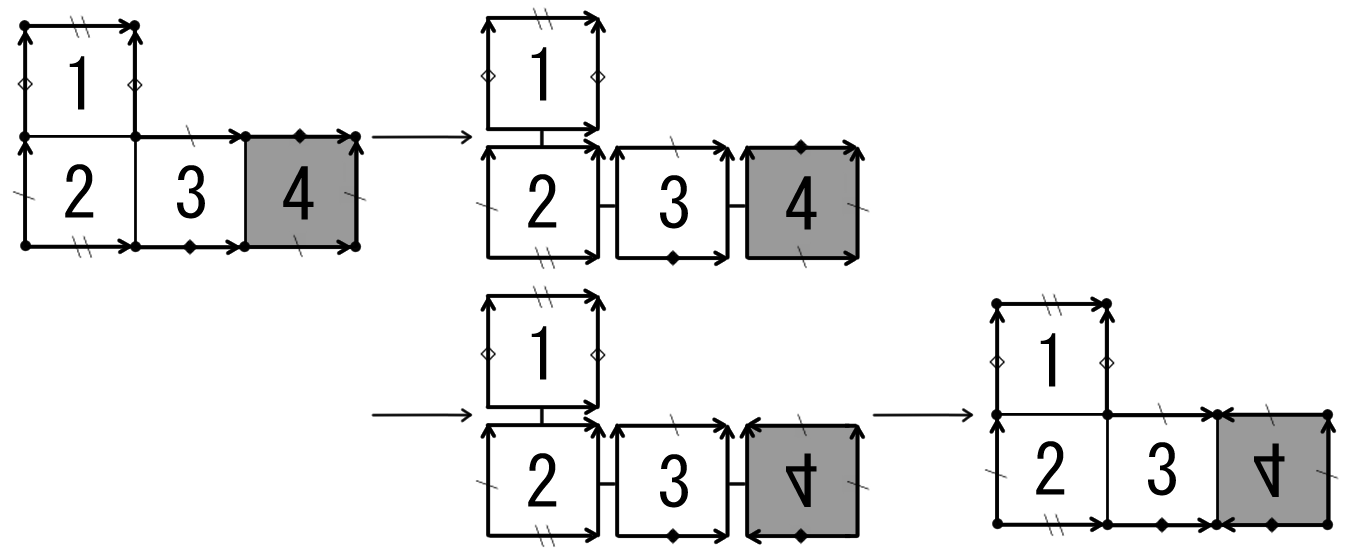}
  \caption{The construction of the origami in Figure \ref{origami_exm}. It is given by $(x,y, \varepsilon)$ where $x=(1)(2\ 3\ 4)$, $y=(1\ 2)(3\ 4)$, $\varepsilon=(+,+,+,-)$. We obtain the origami by regluing the abelian origami $(x,y)$ after inverting squares of negative sign. }
\label{construction}    
\end{center}
\end{figure}

\begin{lemma}[{\cite[Section 4]{K}}]\label{isom}Let $\mathcal{O}_j$ be two origamis of degree $d$ given by $(x_j,y_j,\varepsilon_j)$ 
$(j=1,2)$. Then $\mathcal{O}_1,\mathcal{O}_2$
are isomorphic as flat surfaces if and only if there exists $\delta\in\{\pm1 \}^d$ and $\sigma \in S_d$ 
 such that following hold on $\{1,\ldots ,d\}$.
\begin{enumerate}\setlength{\leftskip}{5pt}\setlength{\itemsep}{3pt}
\item $\delta=\delta\circ x_1$. 
\item $x_2={\sigma}^\# (x_1^{\delta})$. 
\item $\varepsilon\cdot\varepsilon(y_2^\varepsilon)=1$ where $\varepsilon=\delta\circ{\sigma}^{-1}\cdot \varepsilon_1\circ{\sigma}^{-1}\cdot\varepsilon_2$. 
\item $y_2={\sigma}^\#(y_1^{\delta\cdot\varepsilon_1\cdot\varepsilon_2\circ{\sigma}})$. 
\end{enumerate}
In particular, $\theta(x,y,\varepsilon):=(x^\mathrm{sign}, \varepsilon y^\varepsilon\varepsilon(y^\varepsilon))$ defines a 1-1 correspondence between $\Omega_d:=\{(x,y,\varepsilon)\mid x,y\in S_d,\ \varepsilon \in \{\pm1 \}^d \}$ and the set of origamis up to equivalence. 
\end{lemma}

\begin{theorem}\label{gen_origami}An origami of degree $d$ is up to equivalence uniquely determined by each of the following. 
\begin{enumerate}
\item{A $2d$-fold cover $p:R\rightarrow P^1_\mathbb{C}$ with valency list $(2^d\mid 2^d\mid 2^d\mid *\ )$. }
\item{A pair of abelian origami of degree $d$ and a $d$-tuple of signs.}
\item{A connected tripartite graph $(\mathcal{V}=\mathcal{V}_{ c}\sqcup\mathcal{V}_{ h}\sqcup\mathcal{V}_{ v}, \mathcal{E})$ with $|\mathcal{V}_{ c}|=|\mathcal{V}_{h}|=|\mathcal{V}_{v}|=d$ such that each edge connects vertices in $\mathcal{V}_{ c}$ and either $\mathcal{V}_{ h}$ or $\mathcal{V}_{ v}$, and each vertex in $\mathcal{V}_{ c},\mathcal{V}_{ h}, \mathcal{V}_{ v}$ has valency $4,2,2$ respectively. }
\item{A pair of permutations $\mu,\nu\in \mathrm{Sym}\{\pm1$, \ldots, $\pm d\}$ that are fixed-point-free, of order $2$, and together with sign inversion generate a transitive group.}
\end{enumerate}
\end{theorem}
\begin{pf}(origami $\Leftrightarrow$ {\rm (a)} $\Leftrightarrow$ {\rm (b)}) A cover {\rm (a)} uniquely lifts the flat structure on the \textit{pillowcase sphere} obtained by gluing two half-squares back to back. 
The equivalence between origamis and {\rm (b)} follows from Lemma \ref{isom}. 
The construction referring to Lemma \ref{cutpaste} shows that a cover {\rm (a)} is obtained from an origami by inverting some monodromies of a cover {\rm (a)} which is the composition of an abelian origami $p:R'\rightarrow E$ and the elliptic involution quotient $q:E \rightarrow P^1_\mathbb{C}$. 

\noindent
(origami $\Leftrightarrow$ {\rm (c)} $\Leftrightarrow$ {\rm (d)}) A graph {\rm (c)} defines an origami by assigning a unit square cell to each vertex in $\mathcal{V}_c$, a horizontal edge to each vertex in $\mathcal{V}_h$, a vertical edge to each vertex in $\mathcal{V}_h$, and the adjacency between a cell and an edge to each edge in $\mathcal{E}$. 
Conversely, a cover {\rm(a)} and $(C_0=\{y=4x^3-x\}, \beta_0(x,y)=4x^2)$ induce a \textit{dessin d'enfant} on $R$ as a graph {\rm (c)}. 
The rest of proof follows from \cite[Proposition 3.2]{HS1}.
 \qed\end{pf}
Figure \ref{origami_dessin} shows an example of $4d$-fold cover $\beta=\beta_0\circ p:R\rightarrow P^1_\mathbb{C}$ branched over just three points $0,1,\infty \in P^1_\mathbb{C}$.  
The monodromy group of $\beta$ is generated by two permutations $\iota,\sigma=\sigma_{\mu,\nu}\in \mathrm{Sym}(\bar{I}_d^h\sqcup\bar{I}_d^v)$ 
defined by 
\begin{align}
  \left\{
    \begin{array}{ll}
      \iota(\pm \lambda _h)=\pm \lambda _v&\ \iota(\pm \lambda _v)=\mp\lambda  _h\\ 
      \sigma(\pm \lambda _h)=\mu(\pm \lambda )_h&\  \sigma(\pm \lambda_v)=\nu(\pm \lambda )_v
    \end{array}
    \right. \text{\ \ for each $i \in I_d$, }
\end{align}   
where $\bar{I}_d^\bullet=\{\pm1_\bullet,\ldots,\pm d_\bullet\}$ denotes a copy of $\bar{I}_d$. 
The permutation $\iota\sigma$ arranges the sheets of $\beta$ clockwise around each of the centers of cells in $R\setminus \beta^{-1}([0,1])$, which are the singularities of $(R,\phi)$. 
The permutation $\iota\sigma$ has even order, as does that of the pillowcase sphere. 
\begin{figure}[htbp]
\begin{center}
\includegraphics[width=130mm]{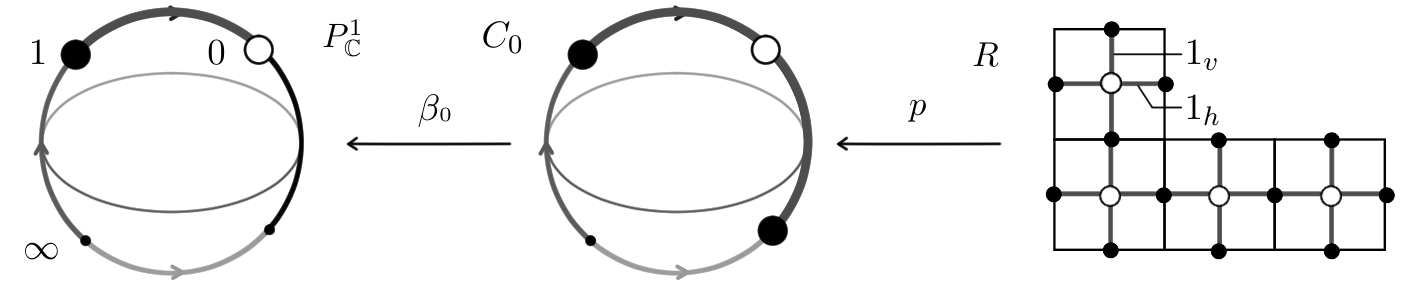}
  \caption{(origami in Figure \ref{origami_exm}, \ref{construction}) 
Suppose the bipartite graph $\beta^{-1}([0,1])$ embedded in $R$. 
The monodromy group of $\beta$ is generated by the two permutation $\iota$, $\sigma$ of edges   around the white, brack vertices respectively. 
Each edge is labeled by the index of the square it belongs to and its direction. 
For example, the horizontal edge adjacent to the right (resp.\ left) side of $\lambda$-th square is labeled by $+\lambda_h$ (resp.\ $-\lambda_h$).}
\label{origami_dessin}
\end{center}
\end{figure}

Note about the equivalences of the objects in Theorem \ref{gen_origami} as follows. 
The equivalence of a graph embedding into the origami as a ﬂat surface gives the equivalence of \textrm{(c)}. 
The conjugacy in $\bar{S}_d$ gives the equivalence of \textrm{(d)}. 

The Veech group of an origami is calculated using the following proposition. 

\begin{proposition}[{\cite[Section 5]{K}}]
There exists two permutations $\sigma_T,\sigma_S\in S_d$ such that the Veech group of an origami 
of degree $d$ is the stabilizer of the equivalence class under the action of $PSL(2,\mathbb{Z})$ on $\Omega_d$ defined by $[A](x,y,\varepsilon):=\theta^{-1}(\sigma_A^*(A\theta(\mathcal{O})))$ $([A]=[T],[S])$. 
\end{proposition}



\section{Natural metric on a flat surface}
\label{sec3-1}
Let $(R,\phi) $ be a flat surface. 
The Euclidian metric lifts via $\phi$-coordinates to a flat metric 
on $R$, called the $\phi$\textit{-metric}. 
A geodesic of $\phi$-metric is called a $\phi$\textit{-geodesic}. 
Via the $\phi$-coordinates, a $\phi${-geodesic} is locally a line segment on the plane whose \textit{direction} is uniquely determined in $\mathbb{R}/\pi\mathbb{Z}$. 

\begin{definition}\ 

\begin{enumerate}
\item The \textit{direction} of a $\phi$-geodesic $\gamma$ is $\theta\in \mathbb{R}/\pi\mathbb{Z}$ where $\gamma$ is horizontal $e^{2\sqrt{-1}\theta}\phi$-geodesic. 
\item The $\phi$\textit{-cylinder} generated by a $\phi$-geodesic $\gamma$ is the union of all $\phi$-geodesics parallel (with same direction) and free homotopic to $\gamma$. We define the direction of a $\phi$-geodesic by the one of its generator. 
\item $\theta\in \mathbb{R}/\pi\mathbb{Z}$ is \textit{Jenkins-Strebel direction} of $(R,\phi)$ if almost every point in $R$ lies on some closed $\phi$-geodesic in the direction $\theta$. 
We denote the set of Jenkins-Strebel directions by $J(R,\phi)$. 
\end{enumerate}
\end{definition}
Note that any Jenkins-Strebel direction of flat surface of finite analytic type is \textit{finite} i.e.\ there are at most finitely many $\phi$-cylinders of that direction in $R$.
We say that a system $\gamma=(\gamma_1$, \ldots, $\gamma_p)$ of Jordan curves on $R$ is \textit{admissible} if none of the curves is homotopically trivial and any two distinct $\gamma_i,\gamma_j$ are neither crossing nor (freely) homotopic. 
For the existence of a holomorphic quadratic differential with one Jenkins-Strebel direction, the following result is known.
\begin{proposition}[{\cite[Theorem 21.1]{St}}]
Let $\gamma=(\gamma_1$, \ldots, $\gamma_p)$ be a finite `admissible' curve system on $R$, which satisfies bounded moduli condition for $\gamma$. Then for any $b=(b_1$, \ldots, $b_p)\in\mathbb{R}_+^p$ {there exists a holomrphic quadratic differential  $\phi$ on $R$ such that $0\in J(R,\phi)$} and $(R,\phi)$ is decomposed into cylinders $(V_1,...,V_p)$ where {each $V_j$ has homotopy type $\gamma_j$ and height $b_j$}. 
\end{proposition}
Let $f\in \mathrm{Aff}^+(R,\phi)$ and $D(f)=[A]=\begin{sbmatrix}a&b\\ c&d\end{sbmatrix}\in PSL(2,\mathbb{R})$. Then $f$ maps any line segment in the direction $\theta\in \mathbb{R}/\pi\mathbb{Z}$ to a line segment in the direction $A\theta:=\mathrm{arg}(T_A(e^{\sqrt{-1}\theta}))$. 
Let $f\in \mathrm{Aff}^+(R,\phi)$ be an affine mapping with derivative $D(f)=[A]=\begin{sbmatrix}a&b\\ c&d\end{sbmatrix}\in PSL(2,\mathbb{R})$. 
Then, $f$ maps any line segment in the direction $\theta\in \mathbb{R}/\pi\mathbb{Z}$ to a line segment in the direction $A\theta:=\mathrm{arg}(T_A(e^{i\theta}))$. 
We may observe that $f$ maps a $\phi$-cylinder of modulus $M$ to a $\phi$-cylinder of modulus $M/\sqrt{a^2+c^2}$. Since the list of moduli of $\phi$-cylinders of one direction are uniquely determined up to order, the following holds. 

\begin{lemma}\label{cyl}
Let $J(R,\phi)\neq\emptyset$ and $(M_1^\theta$, \ldots, $M_{n_\theta}^\theta)\in\mathbb{R}_{>0}^{n_\theta}$ be the list of ascending order of moduli of the $\phi$-cylinders in the direction $\theta\in J(R,\phi)$.
 If $[A]=\begin{sbmatrix}a&b\\ c&d\end{sbmatrix}\in PSL(2,\mathbb{R})$ belongs to $\Gamma (R,\phi)$ then for any $\theta\in J(R,\phi)$ the following holds. 
\begin{itemize}
\item $A\theta\in J(R,\phi)$.
\item $n_{A\theta}=n_\theta=:n$.
\item 
$M_j^{A\theta}=M_j^{\theta}/\sqrt{a^2+c^2}$ for $j=1$, \ldots, $n$. 
\end{itemize}
 \end{lemma}

From now on, we assume that there exist two distinct Jenkins-Strebel directions $\theta_1, \theta_2\in J(R,\phi)$. 
Then $R$ is obtained by finite collections of parallelograms in the way presented in {\cite[Theorem2]{EG}}, in which we conclude $R$ is finite analytic type even for more general settings. 
We review that construction. 

Let $ i=1,2$ and $\alpha_i=e^{\sqrt{-1}\theta_i}\in\mathbb{R}/\pi\mathbb{Z}$. We have a decomposition of $R$ into the $\phi$-cylinders $W_1^i,...,W_{n_i}^i$ in direction $\theta_i$. 
For each $i,j$, an analytic continuation of local inverse of $\phi$-coordinates gives a holomorphic cover $F^i_j:S^i_j\rightarrow W_j^i$ from a strip $S_j^i=\{0<\mathrm{Im}z<h_j^i\}\subset \mathbb{C}$ with $\mathrm{Deck}(F_j^i)=\langle z\mapsto z+c_j^i \rangle$ for some $h_j^i, c_j^i>0$ (See Figure \ref{cov_cyl}). 
We denote by $z_j,w_j$ the local $\phi$-coordinates in $W_j^1$, $W_j^2$. By construction, $F_j^{1*}(\alpha_i\phi)={dz_j}^2$ and $F_j^{2*}(\alpha_i\phi)={dw_j}^2$ hold. 

\begin{figure}[htbp]
\begin{center}
\includegraphics[width=120mm]{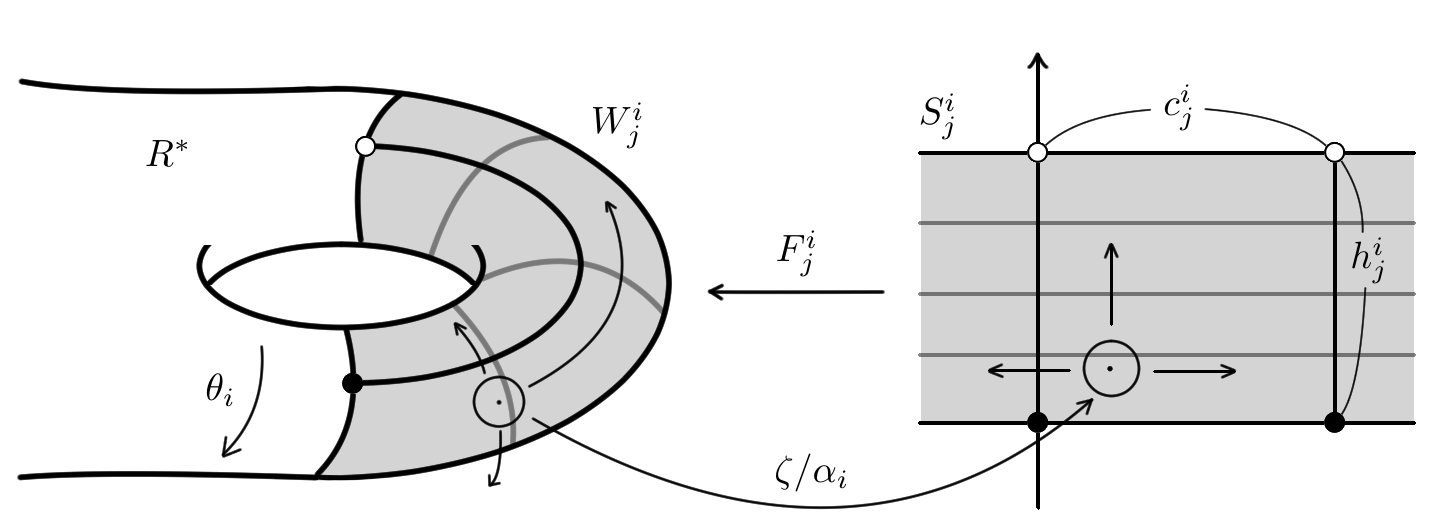}
  \caption{$\phi$-cylinder $W_j^i$ and covering $F^i_j$}
\label{cov_cyl}    
\end{center}
\end{figure}

For any $p\in S_j^1$, there is a neighborhood $U$ in which $F_j^1=F_k^2\circ f$ for some $k$ and some holomorphic function $f:U\rightarrow S_k^2$. 
By the formula 
\begin{equation}f^*{dw_k}^2=f^*(F_k^{2*}(\alpha_2\phi))=F_j^{1*}(\alpha_2\phi)=(\alpha_2/\alpha_1){dz_j}^2
\end{equation}   
$f$ is continuated on $S_j^1$ by the form $f(z_j^1)=\alpha z_j^1+\beta$ where $\alpha=\pm\sqrt{\alpha_2/\alpha_1}$ and $\beta\in \mathbb{C}$ (see Figure \ref{intersection}). 
The intersection $V_{j,k}=S^1_j\cap f^{-1}(S^2_k)$ is a parallelogram isometrically mapped to $W^1_j\cap W^2_k$. The collection $(V_{j,k})_{j=1}^{n_1}$ fills the strip region $S^1_j$ by translations in  $\mathrm{Deck}(F_j^1)=\langle z\mapsto z+c_j^1\rangle$, $j=1$, \ldots, $n_1$. 
The same can be said for $(f^{-1}(V_{j,k}))_{k=1}^{n_2}$ filling the strip region $S_k^2$. 

\begin{figure}[htbp]
\begin{center}
\includegraphics[width=120mm]{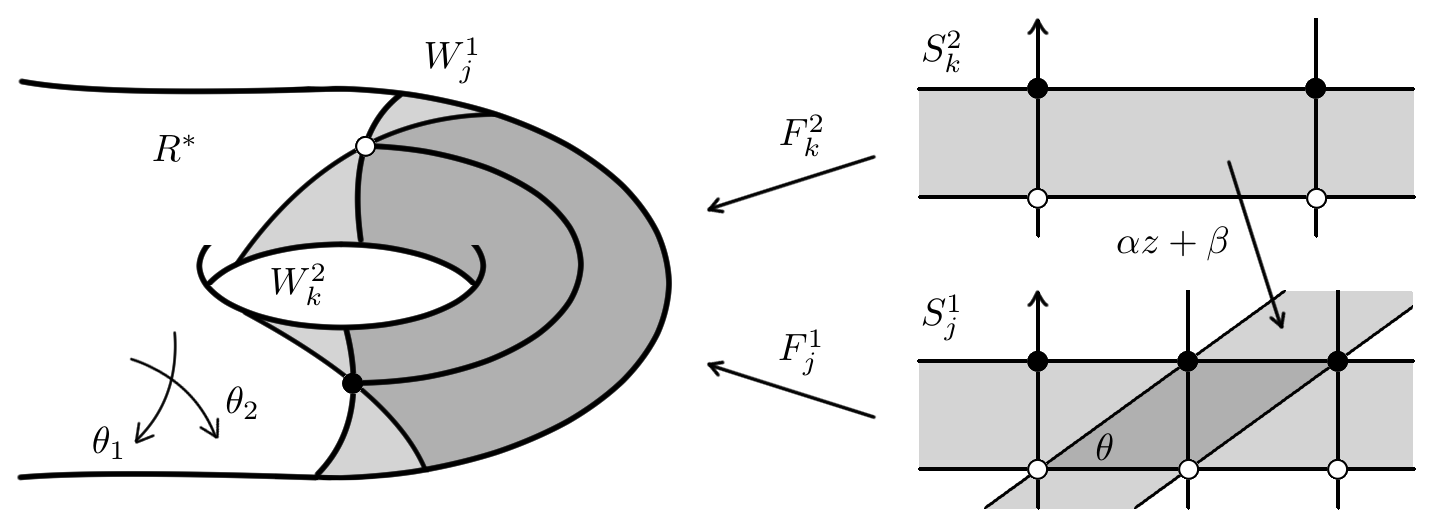}
  \caption{Parallelogram as an intersection of two $\phi$-cylinders}
\label{intersection}    
\end{center}
\end{figure}


Thus the surface $R$ is decomposed into the collection of regions $(W_{j}^1\cap W_{k}^2)_{j,k}$, each of which is empty or isomorphic to a parallelogram on the plane. 
Suppose $(j,k)$ in the latter case. 
Such a parallelogram $V_{j,k}$ is uniquely determined up to half-translations. 
We call them the $(\theta_1,\theta_2)$-\textit{parallelograms} of $(R,\phi)$. 
Via $F_j^1$ and $F_k^2$, the isomorphism between $W_{j}^1\cap W_{k}^2$ and $V_{j,k}$ is continued over the boundary. Thus $(R,\phi)$ is isomorphic to the surface obtained by gluing $(\theta_1,\theta_2)$-parallelograms along boundary edges in the way that respects the adjacencies determined by the continuations of the local isomorphisms. 


A $(\theta_1,\theta_2)$-parallelogram $V_{j,k}$ has boundary edges in the directions $\theta_1,\theta_2$ and a modulus $M(V_{j,k})=(h_{j}^1/h_{k}^2) \sin |\theta_1-\theta_2|$. 
On the plane, an affine map with derivative $A\in SL(2,\mathbb{Z})$ maps a $(\theta_1,\theta_2)$-parallelogram to an $(A\theta_1,A\theta_2)$-parallelogram whose modulus is a scalar multiple of $\rho_{A,\theta_1,\theta_2}:=|T_A(e^{\sqrt{-1}\theta_2})|/|T_A(e^{\sqrt{-1}\theta_1})|$. 
The same holds for $(R,\phi)$ as follows. 

\begin{lemma}\label{paral}
Let $(R,\phi)$ be a flat surface with two distinct Jenkins-Strebel directions $\theta_1,\theta_2$, and $\{V_\lambda\}_{\lambda=1}^d$ be the $(\theta_1,\theta_2)$-parallelograms of $(R,\phi)$. 
If a matrix $[A]\in PSL(2,\mathbb{R})$ belongs to $\Gamma(R,\phi)$, then $M(f(V_j))=\rho_{A,\theta_1,\theta_2}\cdot M(V_1)$ holds for $j=1$, \ldots, $d$. 


\end{lemma}

\medskip\medskip
\section{Combinatorial formulation}\label{sec3-3}
We continue assumptions and notations in the last section. 
Stretching and rotating $\phi$-cylinders lead to a homeomorphism from $(R,\phi)$ to an origami that respects the markings determined by boundaries of parallelograms. 
In this way, $(R,\phi)$ and $\theta_1,\theta_2\in J(R,\phi)$ determine a unique origami with additional data of moduli list $\mathbf{M}=(M_\lambda)_{\lambda=1}^d$ of the $(\theta_1,\theta_2)$-parallelograms and directions $\theta_1,\theta_2$. 
Conversely, an origami $\mathcal{O}$ and a moduli list $\mathbf{M}=(M_\lambda)_{\lambda=1}^d$ compatible with $\mathcal{O}$ is supposed to give a flat surface with decomposition as above for each pair $(\theta_1,\theta_2)$ of distinct directions assigined. 

We use the notation in Theorem \ref{gen_origami}. 
We will define the compatibility of $\mathbf{M}=(M_\lambda)_{\lambda=1}^d\in\mathbb{R}^d_{>0}$ with an origami $\mathcal{O}=(\mu,\nu)$, which purposes that we can glue $d$ rectangles $V_1$, \ldots, $V_d$ with $M(V_\lambda)=M_\lambda$ along edges to form a flat surface $(R,\phi)$ in the same way as $\mathcal{O}$. 


Let $\left| \kappa\right| =\lambda$ for each $\kappa=\pm\lambda_{\bullet} \in\bar{I}_d^h\sqcup\bar{I}_d^v$. 
Then $|\mu(\kappa)|$ (resp.\ $|\nu(\kappa)|$) represents the rectangle adjacent to the right (resp.\ upper) side of $|\kappa|$-th rectangle. 
Then the lengths of their horizontal (resp.\  vertical) edges should be related by a factor of $K_{\kappa,\mu}=M_{|\kappa|}/M_{|\mu(\kappa)|}$ (resp.\ $K_{\kappa,\nu}=M_{|\nu(\kappa)|}/M_{|\kappa|}$). 
When we go along  a path $\gamma$ on ${R}^*$ joining two rectangles, indices of rectangles we pass through and entry directions are interpreted as a path in the bipartite graph $\beta^{-1}([0,1])$. 
It is described in terms of monodromy of the form $\iota^{k_1}\sigma\cdots\iota^{k_m}\sigma\in\bar{S}_d$, which is a word of $\iota^{k}\sigma_k$ $(k=0,1,2,3)$. 
We may set starting edge as $+\lambda_h$, then we have $\sigma_{k_j}=\mu$ for $k_j=0,2$ and $\sigma_{k_j}=\nu$ for $k_j=1,3$. 
We define as follows. 
\begin{align}\label{loop}
K_\mathcal{O}(\gamma,\mathbf{M}):=\prod_{j=1}^{m}K_{\iota^{k_1}\sigma\cdots\iota^{k_{j}}\sigma(+\lambda_h),\ \sigma_{k_j}}
\end{align}   

\begin{definition}\label{compatible}Let $\mathcal{O}=(\mu,\nu)$, $\mathcal{O}_i=(\mu_i,\nu_i)$ be origamis of degree $d$. 
\begin{enumerate}
\item 
We call $\mathbf{M}=(M_\lambda)_{\lambda=1}^d\in \mathbb{R}^{d}_{>0}$ a \textit{moduli list compatible with} $\mathcal{O}$ 
if $K_\mathcal{O}(\gamma,\mathbf{M})=1$ for any $\gamma\in\pi_1(\mathcal{O}^*)$. ($\mathcal{O}^*$ is the flat surface $\mathcal{O}$ punctured at all the corner points.) 
\item Let $\mathbf{M}_i=(M^i_\lambda)_{\lambda=1}^{d}\in \mathbb{R}^{d}_{>0}$ be a moduli list compatible with $\mathcal{O}_i$ for $i=1,2$. 
We say that $(\mathcal{O}_1,\mathbf{M}_1)$ and $(\mathcal{O}_2,\mathbf{M}_2)$ are equivalent if there exists $\tau\in \bar{S}_d$ such that the following holds for $\lambda=1, \ldots, d$. 
\begin{align}
\mu_1=\tau^*\mu_2, \ \nu_1=\tau^*\nu_2, \ M_\lambda^1=M^2_{|\tau(\lambda)|}
\end{align}   
\end{enumerate}
\end{definition}
Observe that an isomorphism between two flat surfaces with two finite Jenkins-Strebel directions naturally induces an equivalence between two origamis with compatible moduli lists. 

\begin{theorem}\label{determined}
Let $\theta_1,\theta_2\in \mathbb{R}/\pi\mathbb{Z}$ be two distinct directions. 
A flat surface $(R,\phi)$ such that $\theta_1,\theta_2 \in J(R,\phi)$ is up to equivalence uniquely determined by an origami 
with a compatible moduli list. 
\end{theorem}
The mapping $K_\mathcal{O}(\hspace{2pt}\cdot\hspace{2pt},\mathbf{M})$ defines a group homomorphism $\pi_1(\mathcal{O}^*)\rightarrow \mathbb{R}_{>0}$. 
The compatibility of lengths for the rectangles placed along a path $\gamma$ on $R^*$ fails only when $\gamma$ contains a loop. 
We may determine the compatibility by choosing a finite generating system of $\pi_1(\mathcal{O}^*)$. 
As the maping $K_\mathcal{O}(\gamma,\cdot\hspace{2pt})$ is regarded as a linear map via the conjugation by the logarithm, we obtain an integer matrix $A_\mathcal{O}$ with $d$ rows representing a linear equation $A_\mathcal{O}\cdot\mathbf{M}=0$ to ensure compatibility. 

For each $(x_1,\ldots x_d)\in \mathrm{Ker}A_\mathcal{O}$, an isomorphism on the flat surface $(\mathcal{O},(e^{x_1},\ldots e^{x_d}))$ induces a cell-to-cell correspondence $\sigma\in \bar{S}_d$ that commutes with $\mu,\nu$. 
It acts by permutating coordinates with indices in the same orbit of the centralizer $C_\mathcal{O}=\mathrm{Cent}_{\bar{S}_d}( \mu,\nu)$.

\begin{corollary}Let  $\theta_1,\theta_2\in \mathbb{R}/\pi\mathbb{Z}$ be two distinct directions and $\mathcal{O}=(\mu,\nu)$ be an origami of degree $d$. 
Then the family of flat surfaces with two finite Jenkins-Strebel directions $\theta_1,\theta_2$ inducing origami $\mathcal{O}$ is globally parametrized in the quotient $\mathrm{Ker}A_\mathcal{O}/C_\mathcal{O}$. 
\end{corollary}

The group $C_\mathcal{O}'=\mathrm{Cent}_{\mathrm{Sym}(\bar{I}_d)}( \mu,\nu )$ is the automorphism group of the (possibly disconnected) dessin $(\mu,\nu)$ of degree $2d$. 
The graph of $(\mu,\nu)$ is the disjoint union of cycle graphs each components of which corresponds to the $\phi$-cylinders in the direction $[\frac{3}{2}\pi]\in \mathbb{R}/\pi\mathbb{Z}$. 
The group $\mathrm{Aut}(\mu,\nu)$ is described by the form
\begin{equation}
C_\mathcal{O}'\cong\left(\prod_{i=1}^{k}\{\pm1 \}\rtimes \mathbb{Z}/l_i \mathbb{Z}\right)\times \left(\prod_{l=1}^{d}S_{n(l)}\right),
\end{equation}
where the permutation $\mu\iota\nu\in\mathrm{Sym}(\bar{I}_d)$ has $k$ cycles of lengths $l_1,\ldots, l_k$ and $n(l)=\#\{i\mid l_i=l\}$. 
We have $C_\mathcal{O}=C_\mathcal{O}'\cap \bar{S}_d$.

\medskip\medskip
\section{Veech group of flat surface with two finite Jenkin-Strebel directions}
For a flat surface with two finite Jenkins-Strebel directions $\theta_1,\theta_2\in J(R,\phi)$, let $P(R,\phi,(\theta_1,\theta_2))$ denote the origamis with compatible moduli lists given by the decomposition of $(R,\phi)$ in $(\theta_1,\theta_2)$. 
\begin{corollary}[to Theorem \ref{determined}]\label{VG_determined}
Let $(R,\phi)$ be a flat surface with two finite Jenkins-Strebel directions $\theta_1,\theta_2\in J(R,\phi)$. 
A matrix $[A]\in PSL(2,\mathbb{R})$ belongs to $\Gamma (R,\phi)$ if and only if the following holds. 
\begin{enumerate}
\item $A\theta_1,A\theta_2$ belong to $ J(R,\phi)$. 
\item Let $P(R,\phi,(\theta_1,\theta_2))=(\mathcal{O},\mathbf{M})$ and $P(R,\phi,A(\theta_1,\theta_2))=(\mathcal{O}_A,\mathbf{M}_A)$. Then, $(\mathcal{O},\mathbf{M})$ is equivalent to $[A^{-1}]\cdot(\mathcal{O},\mathbf{M}):=(\mathcal{O}_A,\rho_{A,\theta_1,\theta_2}^{-1}\cdot\mathbf{M}_A)$. 
\end{enumerate}
\end{corollary}
\begin{proof}
If $f\in \mathrm{Aff}^+(R,\phi)$ of derivative $[A]$ exists, it maps the $(\theta_1,\theta_2)$-parallelograms to the $A(\theta_1,\theta_2)$-parallelograms with their adjacency preserved, and $A\theta_1,A\theta_2$ belong to $ J(R,\phi)$.  
The change of moduli of these parallelograms are described by the multiple of the constant $\rho_{A,\theta_1,\theta_2}$ in Lemma \ref{paral}. 
Thus the mapping $f$ pulls back the decomposition $P(R,\phi,A(\theta_1,\theta_2))$ to the decomposition $P(R,\phi,(\theta_1,\theta_2))$ as equivalent to $[A^{-1}]\cdot(\mathcal{O},\mathbf{M})$ by Theorem \ref{determined}. 
The converse directly follows from Theorem \ref{determined}. 
\end{proof}

\begin{definition}
  Let $(R,\phi)$ and $(S,\psi)$ be flat surfaces with two finite Jenkins-Strebel directions. 
  We say that a finite covering $f:(S,\psi)\rightarrow (R,\phi)$ is \textit{unbranched} if it is locally-affine and $\mathrm{Crit}(f)\subset f^{-1}(\mathrm{Sing}(\phi))\subset \mathrm{Sing}(\psi)$ holds.  
\end{definition}
\begin{rem}
For a locally-affine covering $f:(S,\psi)\rightarrow (R,\phi)$, $\psi=f_*\phi$ holds and local behavior of $\phi,\psi$ naturally correspond under $f$. 
The condition $\mathrm{Crit}(f)\subset f^{-1}(\mathrm{Sing}(\phi))$ implies that $f$ branches at most over the singularities of $(R,\phi)$. 
The condition $f^{-1}(\mathrm{Sing}(\phi))\subset \mathrm{Sing}(\psi)$ implies that no singularity on $(R,\phi)$ is canceled when pulled back. 
Note that the canonical double covering of non-abelian flat surface is not an unbranched covering. 
\end{rem}

For flat surfaces in covering relation, the commensurability of the Veech groups is known \cite{GJ}. 
More strongly, the following holds in our situation. 
\begin{lemma}\label{lem_VG_cov}
  Let $f:(S,\psi)\rightarrow (R,\phi)$ be an unbranched covering of flat surfaces with two finite Jenkins-Strebel directions. 
  Then $\Gamma(S,\psi)$ is a finite index subgroup of $\Gamma(R,\phi)$. 
\end{lemma}
\begin{proof}
  Let $\theta_1,\theta_2\in J(R,\phi)$, $p\in R\setminus\mathrm{Sing}(\phi)$, and $\gamma$ be a closed $\phi$-geodesic in the direction $\theta_1$ through $p$. 
  Then any lift of $\gamma$ is a $\phi$-geodesic joining points in $f^{-1}(p)$ in the direction $\theta_1$.  
  Finite collection of such lifts form a closed $\psi$-geodesic and any closed $\psi$-geodesic is of this form. 
  Since $\psi=f_*\phi$ where no singularity on $(R,\phi)$ is canceled, any pullbacks of a $\phi$-cylinder are not laminated together to make a wider cylinder.
  Thus $f$ induces 1-1 correspondence between $J(R,\phi)$ and $J(S,\psi)$, and between $(\theta_1,\theta_2)$-parallelograms of $(R,\phi)$ and $(S,\psi)$.   

  Let $\mathcal{O}_R$ ($\mathcal{O}_S$, respectively) denote the origami determined by the decomposition $P(R,\phi,(\theta_1,\theta_2))$ ($P(S,\psi,(\theta_1,\theta_2))$, respectively). 
  We can see that $\mathcal{O}_S$ is obtained from finite copies of $\mathcal{O}_R$ by regluing along their edges according to the monodromy of $f$. 
  Furthermore $f$ induces a projection from $\mathcal{O}_S$ to $\mathcal{O}_R$ which respects adjacency of squares up to the copies. 
  So if $(S,\psi)$ satisfies the condition {\rm (b)} in Corollary \ref{VG_determined}, then the same holds for $(R,\phi)$.  
  Conversely, for $[A]\in \Gamma(R,\phi)$, the origami determined by $(S,\psi)$ with $A(\theta_1,\theta_2)$ is similarly constructed as $\mathcal{O}_S$ up to difference of monodromy. 
  As it has finitely many possibilities, it coincides with $\mathcal{O}_S$ up to finite representatives. 
  The same can be said for the decomposition of $(R,\phi)$ into parallelograms. 
\end{proof}

\begin{theorem}\label{VG_covering}
  Let $f:(S,\psi)\rightarrow (R,\phi)$ be an $N$-fold, unbranched covering of flat surfaces with $(\theta_1,\theta_2)\in J(R,\phi)$. 
  Fix a base point $p\in R^*$ and a finite generating system $\mathcal{F}$ of $\pi_1(R^*,\cdot)$. 
  Let the Veech group $\Gamma(R,\mu)<PSL(2,\mathbb{R})$ act on $\mathcal{M}=({S}_N)^{\mathcal{F}}$ in  the way that a matrix $[A]\in\Gamma(R,\mu)$ transforms the monodromy of $f$ by taking the new decomposition in $A^{-1}(\theta_1,\theta_2)$. 
  Then, the Veech group $\Gamma(\hat{R},\psi)$ is the stabilizer of $\tau_f=m_f(\mathcal{F})\in\mathcal{M}$ under the equivalence defined by 
  \begin{enumerate}
    \item (relabeling of sheets of $f$) conjugation in $S_N$ and 
    \item (automorphism of $(R,\phi)$) conjugation in $C_{\mathcal{O}_R}$ where $\mathcal{O}_R$ is the origami determined by the decomposition $P(R,\phi,\theta_1,\theta_2)$. 
  \end{enumerate} 
\end{theorem}
\begin{proof}
  As we have seen in the proof of Lemma \ref{lem_VG_cov}, it follows that any lift of $\gamma\in\pi_1(R^*,\cdot)$ connects two copies of  $P(R,\phi,A^{-1}(\theta_1,\theta_2))=P(R,\phi,(\theta_1,\theta_2))$ in $P(S,\psi,A^{-1}(\theta_1,\theta_2))$ for each $[A]\in \Gamma(R,\phi)$. 
  One obatins the decomposition $P(S,\psi,A^{-1}(\theta_1,\theta_2))$ by patching copies of $P(R,\phi,(\theta_1,\theta_2))$ according to the new monodromy data  $[A]\cdot\tau_f$. 
  It follows from Corollary \ref{VG_determined} that the stabilizer represents the Veech group. 
\end{proof}

\begin{exm}
 The non-abelian origami $\mathcal{D}=((1,2,3,4,5,6)$, $(1,2,5,6,3,4)$, $(-,+,-,+,-,+))$ in Figure \ref{D-monodromy} is the unique nontrivial origami with the maximal Veech group $PSL(2,\mathbb{Z})$ of the smallest degree $6$. 
  Let $f:\mathcal{O}\rightarrow \mathcal{D}$ be an $N$-fold, unbranched covering of the origami $\mathcal{D}$. 
  Fix a finite generating system of $\pi_1(\mathcal{D}^*,\cdot)$ and label the cells of $\mathcal{D}$ as shown in Figure \ref{D-monodromy}. 
  Then the monodromy $\tau_f$ of $f$ can be seen as an element of $\mathcal{M}=({S}_N)^7$ such that the monodromies arround the three poles of $\mathcal{D}$ are not of order two. 
  Automorphisms of $\mathcal{D}$ acts as $C_{\mathcal{D}}=\langle\delta:=(1\ 3\ 5)(2\ 4\ 6)\rangle \cong C_3$, and the action of the Veech group $\Gamma(\mathcal{D})=PSL(2,\mathbb{Z})=\langle[T],[S]\rangle$ on $\mathcal{M}$ is given by
   \begin{figure}[htbp]
   \begin{center}
     \includegraphics[width=90mm]{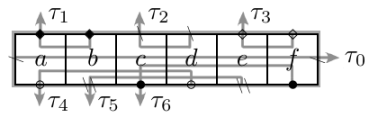}
     \caption{The  origami $\mathcal{D}:(x,y,\varepsilon)=((1,2,3,4,5,6)$, $(1,2,5,6,3,4)$, $(-,+,-,+,-,+))$
: edges with the same marks are glued. 
Arrows are fixed generators of $\pi_1(\mathcal{D}^*,\cdot)$ and $\tau_1,...,\tau_6$ denote monodromies along them. 
}
   \label{D-monodromy}
   \end{center}
   \end{figure}
   \begin{align}
     [T]\cdot\tau&=(\tau_0,\tau_1,\tau_2,\tau_3,\tau_5,\tau_6,{\tau_4^{-1}\tau_0^{-1}}),\label{TD}\\
      [S]\cdot\tau&=(\tau_4\tau_2^{-1}\tau_6\tau_3^{-1}\tau_5^{-1}\tau_1^{-1},
\tau_1\tau_5\tau_3\tau_6^{-1}\tau_2\tau_6\tau_3^{-1}\tau_5^{-1}\tau_1^{-1},\nonumber\\
&\tau_1\tau_5\tau_3\tau_3^{-1}\tau_5^{-1}\tau_1^{-1},
\tau_1, 
\tau_1\tau_5\tau_3\tau_6^{-1}\tau_2\tau_5^{-1}\tau_1^{-1},
\tau_1\tau_5\tau_3\tau_6^{-1}\tau_1^{-1},
\tau_1\tau_5\tau_3\tau_0),
\label{SD}
   \end{align}
   for each $\tau=(\tau_0,\tau_1,\tau_2,\tau_3,\tau_4,\tau_5,\tau_6)\in\mathcal{M}$. 
 Theorem \ref{VG_covering} states that the Veech group $\Gamma(\mathcal{O})$ is the stabilizer of the class of $\tau_f$ under this action.  It can be calculated using the Reidemeister-Schreier method like abelian origamis \cite{S1}. 

   
We obtain the formulae (\ref{TD}) and (\ref{SD}) in the way stated in Figure \ref{TD-monodromy} and Figure \ref{SD-monodromy}. 
  By Theorem \ref{determined}, the covering $f:\mathcal{O}\rightarrow \mathcal{D}$ is uniquely determined by a monodromy $\tau_f\in ({S}_N)^7$ with two directions $(0,\frac{\pi}{2})$. 
   For matrices $[A]=[T],[S]\in \Gamma(\mathcal{D})=PSL(2,\mathbb{Z})$, the decomposition $P(\mathcal{O},A(0,\frac{\pi}{2}))$ is tiled by $[A^{-1}]\cdot\mathcal{D}=P(\mathcal{D},A(0,\frac{\pi}{2}))\cong \mathcal{D}$ as shown in    Figure \ref{TD-monodromy} and Figure \ref{SD-monodromy}. 
   It is also uniquely determined by a suitable monodromy $[A^{-1}]\cdot\tau_f\in ({S}_N)^7$ up to equivalence. 

\end{exm}
  \begin{figure}[htbp]
   \begin{center}
     \includegraphics[width=65mm]{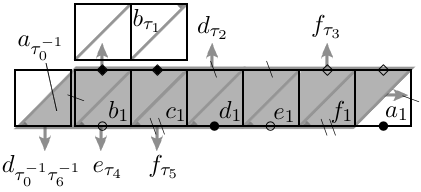}
     \caption{The decomposition $P(\mathcal{O},T(0,\frac{\pi}{2}))$: tiled by $[T^{-1}]\cdot\mathcal{D}\cong \mathcal{D}$ (shaded). 
Each cell of $[T^{-1}]\cdot\mathcal{D}$ is labeled 
according to the sheet of $\mathcal{D}$ to which the left-lower corner belongs. 
     We obtain the formula $[T^{-1}]\cdot \tau=(\tau_0,\tau_1,\tau_2,\tau_3,{\tau_0^{-1}\tau_6^{-1},\tau_4,\tau_5})$ and (\ref{TD}). }
   \label{TD-monodromy}
   \end{center}
   \end{figure}
   \begin{figure}[htbp]
   \begin{center}
     \includegraphics[width=90mm]{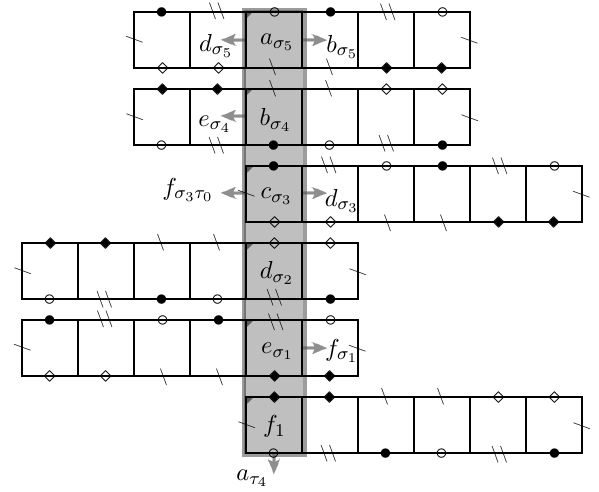}
     \caption{The decomposition $P(\mathcal{O},S(0,\frac{\pi}{2}))$: tiled by $[S^{-1}]\cdot\mathcal{D}\cong \mathcal{D}$ (shaded).  
Each cell of $[S^{-1}]\cdot\mathcal{D}$ is labeled 
according to the sheet of $\mathcal{D}$ to which the left-lower corner belongs. 
The base sheet is denoted by $s=(a_{\sigma_5},b_{\sigma_4},c_{\sigma_3},d_{\sigma_2},e_{\sigma_1},f_1)$ where $\sigma_1=\tau_1$, $\sigma_2=\sigma_1\tau_5$, $\sigma_3=\sigma_2\tau_3$, $\sigma_4=\sigma_3\tau_6^{-1}$, and     $\sigma_5=\sigma_4\tau_2$. 
Starting from the base sheet $s$, the $S$-deformation of $\tau_0$ reaches to the sheet including $a_{\tau_4}$ i.e.\ the sheet $\tau_4\sigma_5^{-1}s$. Thus the sheet transition is $[S^{-1}]\cdot \tau_0=\tau_4\sigma_5^{-1}$. 
 And so on, we obtain the formula $[S^{-1}]\cdot \tau=
(\tau_4\sigma_5^{-1},\sigma_5\sigma_4^{-1},\sigma_3\sigma_2^{-1},\sigma_1 ,\sigma_5\sigma_2^{-1},\sigma_4\sigma_1^{-1},\sigma_3\tau_0)$ and  (\ref{SD}).}
   \label{SD-monodromy}
   \end{center}
   \end{figure}

\newpage
\begin{acknowledgements}
I would like to thank Prof.\ Toshiyuki Sugawa for his helpful advices and comments. 
I am grateful to Prof.\ Hiroshige Shiga for his thoughtful guidance. 
I thank Prof. Rintaro Ohno for several suggestions. 
Some proposals given by Prof.\ Yoshihiko Shinomiya helped me to get an idea for this paper. 
\\
In Weihnachtsworkshop 2019 at Karlsruhe, I had a lot of significant discussions on my research. I would like to thank Prof.\ Frank Herrlich, Prof.\ Gabriela Weitze-Schmith\"usen, and Prof.\ Martin M\"oller for their expert advices and comments. 
I thank Sven Caspart for helpful discussions. 
\\
\\
\end{acknowledgements}

%
%

\end{document}